\documentclass[a4paper, 12pt, notitlepage]{article}
\usepackage{booktabs}
\usepackage[utf8]{inputenc}
\usepackage[english]{babel}
\usepackage{natbib}
\setcitestyle{authoryear,open={(},close={)}} %Citation-related commands
\usepackage{psfrag}
\usepackage{graphicx}
\usepackage{epsfig}
\usepackage{amssymb}
\usepackage{commath}
\usepackage{amsmath}
\usepackage{enumerate}
\usepackage[margin=1in]{geometry}
\usepackage{fancyhdr}
\usepackage{amsfonts}
\usepackage{color}                    % For creating coloured text and background
\usepackage{hyperref}
\usepackage{latexsym}
\usepackage{subfigure}
\newenvironment{proof}[1][Proof]{\textbf{#1.} }{\ \rule{0.5em}{0.5em}}
\newtheorem{definition}{Definition}
\newtheorem{theorem}{Theorem}

\newtheorem{proposition}{Proposition}
\title{A Caputo based SIRS and SIS fractional order models with standard incidence rate and varying population} % change this
\author{\and E. Okyere\thanks{Department of Mathematics and Statistics, UENR, Sunyani, Ghana, {\tt eric.okyere@uenr.edu.gh}}
\and  J. Ackora-Prah\thanks{Department of Mathematics, KNUST, Kumasi, Ghana, {\tt joseph-prah.cos@knust.edu.gh}}\and F. T. Oduro\thanks{African Institute for Mathematical Sciences, Ghana, {\tt francis@aims.edu.gh}} \thanks{Department of Mathematics, KNUST, Kumasi, Ghana, {\tt  ftoduro@gmail.com}}}

\begin{document}
%\pagenumbering{roman}
%\tableofcontents
%%%%%%%%%% PRELIMINARY MATERIAL %%%%%%%%%%
\maketitle
\begin{abstract}
In the present work, we have introduced and studied two epidemic models that are constructed with Caputo fractional derivative. We considered standard incidence rate and varying population dynamics for SIS and SIRS mathematical models. Model equilibria, basic reproduction numbers are determined and local stability analysis are established for the two fractional dynamical models. Finally, the efficient fracPECE iterative scheme for fractional order deterministic dynamical models is applied to perform numerical simulations.
\end{abstract}


\section{Introduction}
Mathematical models have played role in understanding infectious diseases epidemiology \citep{wow11}. Epidemiological models formulated with system of nonlinear fractional order equations have been studied by several authors and it is still gaining a lot of attention in mathematical biology [see, e.g, \cite{Atangan2019stability, pinto2019diabetes, wang2019stability, jajarmi2019new, ameen2017solution, rihan2019fractional, jan2019modeling, Arafa2012, almeida2018analysis}]. An SEIR compartmental model formulation that describes the dynamical behaviour of an influenza A infection is introduced and analyzed by the authors in \citep{Gonzalez2014} using nonlinear fractional order equations. A co-infection epidemic model for malaria and HIV/AIDS infections with fractional order derivative that captures time delay dynamics is studied \citep{Pinto2016}. A mathematical model that describes the dynamics of an HIV/AIDS infection using Caputo based fractional derivative is studied \citep{silva2019stability}.\\

A Caputo fractional derivative is used to construct a nonlinear epidemic model that captures vertical transmission dynamics \citep{OZLAP2011}.
A deterministic model with fractional order formulation is presented to study the dynamical behavior between viral particles, humoral immune response mediated by the antibodies and susceptible host cells \citep{Abena2019}. In the work of \citep{wow14}, the authors derived and studied a non-integer order mathematical model for Hepatitis C disease. \cite{wow30} formulated and analyzed a nonlinear fractional order initial value problem for the co-infection dynamics of Hepatitis C and HIV. A compartmental model for Hepatitis B infection that includes cure of infected cells is presented and analyzed using system of fractional order nonlinear equations \citep{Salman2017}. A non-standard finite discretization scheme is applied to solve two nonlinear Hepatitis B infection models that are of fractional orders \citep{ullah2018new}. Fractional order deterministic models based on an MSEIR compartmental formulation have recently been introduced in the works done by the authors in  \citep{qureshi2019fractional, almeida2019}\\

\cite{ELSAKAfractional} developed and analyzed an SIS model with bilinear incidence rate and varying populaton dynamics. He formulated the epidemic model using differential equations based on Caputo derivative. The same author in \citep{ELSAKA2013} introduced a Caputo based SIRS and SIR models with varying populaton dynamics and bilinear incidence rates. A fractional SIS model problem with constant population and standard incidence rate has been introduced and studied in \citep{FRACSIShassouna2018}. \cite{el2019dynamical} has derived a fractional order Susceptible-Infected-Recovered-Susceptible (SIRS) model based on homogeneous networks. They applied global stability concepts in the work by \citep{vargas2015volterra} and established global stability for their constructed epidemic model. A fractional dynamical problem has been derived and analyzed for the classic SIR endemic model with constant population size \citep{Okyere2016a}. The authors in \citep{liu2019stability} have derived and analyzed an SIS epidemic model based on fractional dynamical nonlinear equations and complex networks. The recent work by the authors in \citep{mouaouine2018fractional} deals  with mathematical formulation of an SIR model with varying population size, nonlinear incidence and Caputo fractional derivative.\\

The fractional order formulation approach we consider for the SIS and SIRS models in our present study is motivated by work done in \citep{wow9}.  The  modeling concept introduced in \citep{wow9} has been used by many authors \citep[see, e.g,][]{ wow30, wow7, Pinto2016, Okyere2016a}. In this study, we will modify and analyze an SIS mathematical model introduced in \citep{vargas2015volterra}. We will also propose a new SIRS model with fractional order dynamics.\\

The outline for the rest of the work is organized as follows: Section~\ref{supereric} is concern with some useful preliminaries on fractional calculus. We present the first mathematical model which is based on an SIS compartmental modeling formulation with fractional order dynamics in section~\ref{onesec}. We then present an SIRS model with fractional order dynamics in section~\ref{foursec}. Local stability analysis, numerical simulations and discussions are presented for the two dynamical models in their respective sections. We then conclude the study in section~\ref{consec}

\newpage
\section{Preliminaries}\label{supereric}
In this work, our modeling formulation is motivated by well-known Caputo derivative in fractional calculus. Several applications of this fractional derivative in engineering and science have been studied in the text books by the authors in \citep{wow2, wow}.

\begin{definition} \citep{wow}
Fractional integral of order $\alpha$ is defined as
\[I^{\alpha}v(t)=\frac{1}{\Gamma(\alpha)}\int_0^t \! \frac{v(x)}{(t-x)^{1-\alpha}} \, \mathrm{d}x\]
\end{definition}
for $0<\alpha<1,\ t>0.$

\begin{definition} \citep{wow}
Caputo fractional derivative of order $\alpha$ is defined as
\[ D^{\alpha}v(t)=\frac{1}{\Gamma(m-\alpha)}\int_0^t \! \frac{v^{m}(x)}{(t-x)^{\alpha+1-m}} \, \mathrm{d}x.\]
\end{definition}
for $m-1<\alpha<m.$

\section{SIS model with fractional order dynamics}\label{onesec}
In this section, we modify and study an existing SIS model that is characterised by system of nonlinear fractional order equations with standard incidence rate. This type of deterministic compartmental structure (susceptible-infected-susceptible) is used to study the dynamical behaviour of infections that do not confer immunity. The total population is grouped into two sub-populations namely susceptible individuals, $Q_S$ and infected individuals, $Q_I$. The fractional order dynamical model presented and analysed in \citep{vargas2015volterra} is given by the nonlinear system as follows,

\begin{align}\label{equation}
D_{t}^{\alpha}Q_S(t) &= \Lambda- \dfrac{\varphi Q_S Q_I}{Q_S+ Q_I}-\nu Q_S+\omega Q_I\\
D_{t}^{\alpha}Q_I(t) &= \dfrac{\varphi Q_S Q_I}{Q_S+ Q_I}-(\eta+\nu+\omega)Q_I\nonumber
\end{align}

The positive model parameters which includes $\Lambda,\ \beta,\ \nu,\ \eta $ and $\omega$ represent recruitment rate of susceptible individuals corresponding to births and immigration, infection rate, natural death rate, disease induced death rate and the rate at which infectious individuals return to the susceptible class after infectious period respectively.\\

From the nonlinear system~(\ref{equation}), it is not hard to see that, there is a mismatch of time dimension on both sides of the system of equations. The left-hand side has a time dimension of $(time)^{-\alpha}$ and the time dimension of the right-hand side is $(time)^{-1}$. The fractionalization method introduced and studied in \citep{wow9} rectifies this drawback of time dimension.

\newpage
Therefore motivated by the modeling approach in \citep{wow9}, the new modified epidemic model which is an improved version of the fractional dynamical initial value problem~(\ref{equation}) is given by

\begin{align}\label{model1}
D_{t}^{\alpha}Q_S(t) &= \Lambda^{\alpha}- \dfrac{\varphi^{\alpha}Q_S Q_I}{Q_S+Q_I}-\nu^{\alpha}Q_S+\omega^{\alpha}Q_I\\
D_{t}^{\alpha}Q_I(t) &= \dfrac{\varphi^{\alpha}Q_S Q_I}{Q_S+Q_I}-(\eta^{\alpha}+\nu^{\alpha}+\omega^{\alpha})Q_I\nonumber
\end{align}

Knowing that $N=Q_{S}+Q_{I}$, the fractional differential equation describing the dynamics of the total population is given by

\begin{align}\label{modeltotal}
D_{t}^{\alpha}N &= \Lambda^{\alpha}- \eta^{\alpha} Q_{I}-\nu^{\alpha} N
\end{align}

Since the right-hand side of equation~(\ref{modeltotal}) is not equal to zero, it implies that, the total population may vary in time.

\subsection{Model equlilibria and local stablility analysis}\label{AbenaChina}
There are two equilibrium points for the fractional dynamical system~(\ref{model1}) \\ given by
$H_{df}=\bigg(\dfrac{\Lambda^{\alpha}}{\nu^{\alpha}},0\bigg)$ and $H_{en}=\bigg({Q_S}^{*}, {Q_I}^{*}\bigg)$\\

where

 \begin{align*}
 {Q_S}^{*}&=\dfrac{\Lambda^{\alpha}}{\nu^{\alpha}+\big(\eta^{\alpha}+\nu^{\alpha}\big)\big(R_0-1\big)};\\\\
 {Q_I}^{*}&= \dfrac{\big(R_0-1\big)\Lambda^{\alpha}}{\nu^{\alpha}+\big(\eta^{\alpha}+\nu^{\alpha}\big)\big(R_0-1\big)}.
 \end{align*}

$H_{df}$ and $H_{en}$ represent disease-free and endemic equilibrium points respectively.\\

The basic reproduction number $R_0$ associated with the model problem~(\ref{model1}) is given as

$$R_0=\dfrac{\varphi^{\alpha}}{\eta^{\alpha}+\nu^{\alpha}+\omega^{\alpha}}.$$\\

At the endemic equilibrium, we obtain the following two identities given by
\begin{align}
\dfrac{\varphi^{\alpha}{Q_S}^{*}}{{Q_S}^{*}+{Q_I}^{*}}&=\eta^{\alpha}+\nu^{\alpha}+\omega^{\alpha}\label{eric}\\
\Lambda^{\alpha} &= (\eta^{\alpha}+\nu^{\alpha}){Q_I}^{*}+\nu^{\alpha}{Q_I}^{*}\label{eric2}
 \end{align}

\begin{theorem}
The equilibrium point $H_{df}$ (disease-free) of the SIS model~(\ref{model1}) is locally asymptotically stable if
$R_{0}<1$ and unstable if $R_{0}>1.$
\end{theorem}

\begin{proof}
The Jacobian matrix of the dynamical fractional SIS model~(\ref{model1}) computed at $H_{df}$ is given by
\[J(H_{df})=
\begin{bmatrix}
-\nu^{\alpha} & -\varphi^{\alpha}+\omega^{\alpha}\\\\
0 & \varphi^{\alpha}-\big(\eta^{\alpha}+\nu^{\alpha}+\omega^{\alpha}\big)
\end{bmatrix}
\]

For the equilibrium point $H_{df}$ to be locally asymptotically stable, it is necessary and sufficient to verify  that the all the eigenvalues of matrix $J(H_{df})$ satisfy the stability condition derived and analyzed in \citep{Ahmed2006}.

\begin{equation}\label{condition}
 \abs{arg(\lambda_{i })} >\frac{\alpha \pi}{2}
 \end{equation}

Since the $2\times 2$ matrix $J(H_{df})$ is a diagonal matrix, it is not hard to see that the eigenvalues  are $\lambda_1= -\nu^{\alpha}<0$ and $\lambda_{2}= \varphi^{\alpha}-\big(\eta^{\alpha}+\nu^{\alpha}+\omega^{\alpha}\big)$. For the two eigenvalues to satisfy the stability condition (\ref{condition}), it is sufficient to show that $\lambda_2$ has negative real part.\\

\begin{align*}
\lambda_{2}&= \varphi^{\alpha}-\big(\eta^{\alpha}+\nu^{\alpha}+\omega^{\alpha}\big)\\
 &=\big(\eta^{\alpha}+\nu^{\alpha}+\omega^{\alpha}\big)\bigg[\dfrac{\varphi^{\alpha}}{\big(\eta^{\alpha}+\nu^{\alpha}+\omega^{\alpha}\big)}-1\bigg]\\
 &=\big(\eta^{\alpha}+\nu^{\alpha}+\omega^{\alpha}\big)\big[R_{0}-1\big] <0 \hspace{0.3cm} iff \hspace{0.3cm} R_{0} <0
\end{align*}

Knowing that $\lambda_1=-\nu^{\alpha}<0$ and following the basic manipulations on the second eigenvalue above, we can conclude that both eigenvalues satisfy the stability condition~(\ref{condition}).
\end{proof}\\

\begin{theorem}
The equilibrium point $H_{en}$ of the SIS model~(\ref{model1}) is locally asymptotically stable if
$R_{0}>1$.
\end{theorem}

\begin{proof}
The Jacobian matrix of the dynamical fractional SIS model~(\ref{model1}) computed at $H_{en}$ is given by

\[J(H_{en})=
\begin{bmatrix}
 -\bigg(\nu^{\alpha}+\frac{\varphi^{\alpha}({Q_I}^{*})^{2}}{\big({Q_S}^{*}+{Q_I}^{*}\big)^{2}}\bigg)&&&\frac{\varphi ({Q_S}^{*})^{2}}{\big({Q_S}^{*}+{Q_I}^{*}\big)^{2}}+\omega^{\alpha}\\\\\\
\frac{\varphi^{\alpha}({Q_I}^{*})^{2}}{\big({Q_S}^{*}+{Q_I}^{*}\big)^{2}}&&& \frac{\varphi^{\alpha}({Q_S}^{*})^{2}}{\big({Q_S}^{*}+{Q_I}^{*}\big)^{2}}-\big(\eta^{\alpha}+\nu^{\alpha}+\omega^{\alpha}\big)
\end{bmatrix}\]

From $\dfrac{\varphi^{\alpha}({Q_S}^{*})}{{Q_S}^{*}+{Q_I}^{*}}=\eta^{\alpha}+\nu^{\alpha}+\omega^{\alpha}$, matrix $J(H_{en})$ becomes\\

\[J(H_{en})=
\begin{bmatrix}
-\nu^{\alpha}-\frac{\varphi^{\alpha}({Q_I}^{*})^{2}}{\big({Q_S}^{*}+{Q_I}^{*}\big)^{2}}&&&
\frac{\varphi^{\alpha}{Q_S}^{*}{Q_I}^{*}}{\big({Q_S}^{*}+{Q_I}^{*}\big)^{2}}-\big(\eta^{\alpha}+\nu^{\alpha}\big)\\\\\\
\frac{\varphi^{\alpha}({Q_I}^{*})^{2}}{\big({Q_S}^{*}+{Q_I}^{*}\big)^{2}}&&& -\frac{\varphi^{\alpha}{Q_S}^{*}{Q_I}^{*}}{\big({Q_S}^{*}+{Q_I}^{*}\big)^{2}}
\end{bmatrix}\]

\newpage
For the two eigenvalues to the satisfy the stability condition~(\ref{condition}), derived and studied by the authors in \citep{Ahmed2006},
it is necessary and sufficient to verify that the Routh-Hurwitz stability conditions for the characteristics equation of $J(H_{en})$ are satisfied.\\

The two eigenvalues of the $2\times 2$ matrix, $J(H_{en})$ can obtained from the characteristic equation given by

\begin{equation}\label{chart}
\lambda^2+a_1 \lambda+a_2 =0
\end{equation}

where

\begin{align*}
a_{1}&= \nu^{\alpha}+\dfrac{\varphi^{\alpha} {Q_I}^{*}}{\big({Q_S}^{*}+{Q_I}^{*}\big)}>0\\
a_{2}&=\frac{\nu^{\alpha}\varphi^{\alpha}{Q_S}^{*}{Q_I}^{*}}{\big({Q_S}^{*}+{Q_I}^{*}\big)^{2}} +\big(\eta^{\alpha}+\nu^{\alpha}\big) \frac{\varphi^{\alpha}({Q_I}^{*})^{2}}{\big({Q_S}^{*}+{Q_I}^{*}\big)^{2}}\\
&=\dfrac{\varphi^{\alpha}{Q_I}^{*}}{\big({Q_S}^{*}+{Q_I}^{*}\big)^{2}}\bigg[\nu^{\alpha}{Q_I}^{*}+\big(\eta^{\alpha}+\nu^{\alpha}\big){Q_I}^{*}\bigg]
\end{align*}

Also from $\Lambda^{\alpha} = \nu^{\alpha}{Q_I}^{*}+\big(\eta^{\alpha}+\nu^{\alpha}\big){Q_I}^{*} $, $a_2$ can further be simplified as
\begin{align*}
a_{2}=\dfrac{\varphi^{\alpha}{Q_I}^{*}}{\big({Q_S}^{*}+{Q_I}^{*}\big)^{2}} \Lambda^{\alpha} >0
\end{align*}

Since the Routh-Hurwitz stability conditions ($a_1>0, \ a_2>0,$) have been verified, it follows that the two eigenvalues will have negative real parts.
\end{proof}

\subsection{Numerical simulations and discussions}
In this subsection, we present numerical solutions of the SIRS model problem~(\ref{model1}),\\  using the well-known and efficient fracPECE iterative scheme introduced in \citep{Diethelm1999}.  For this purpose, we have used the Matlab code fde12.m designed in \citep{Garrappa2011} for the implementation of the fracPECE iteative scheme \citep{Diethelm1999}. Recently the authors in \citep{silva2019stability, kilicman2018fractional, Okyere2016a} have applied this numerical scheme to solve their nonlinear problems.\\

For our numerical experiments, we considered four different values of $\alpha$ ($0.90,\ 0.95,\ 0.99,\ 1$).  The positive values for the model parameters $\Lambda$ and $\nu$ were adapted from work of the author in \citep{vargas2015volterra}. To demonstrate asymptotic stability of the disease-free equilibrium, we have used the following parameter values: $\Lambda=0.01,\  \beta=0.06,\ \\ \nu=0.01,\ \omega=0.02,\ \eta=0.2$ and initial conditions $Q_{S0}=0.95,\ Q_{I0}=0.05$. Asymptotic stability of the endemic equilibrium is illustrated using the following positive parameter values: $\Lambda=0.01,\ \ \beta=0.45,\ \nu=0.01,\  \omega=0.2,\ \eta=0.05$ and initial conditions $Q_{S0}=0.95,\ Q_{I0}=0.05$. Figures~\ref{fg1} and \ref{fg2} represent trajectories of susceptible and infected individuals respectively illustrating asymptotic stability of the disease-free equilibrium. Local asymptotic stability of the endemic equilibrium is demonstrated in Figures~\ref{fg4} and \ref{fg5}. The sub-plots in Figures~\ref{fg3} and \ref{fg6} demonstrate SI plane phase portraits for the SIS model with fractional order dynamics.

\newpage
\begin{figure}[!htbp]
	\centering
	\includegraphics[scale=0.7]{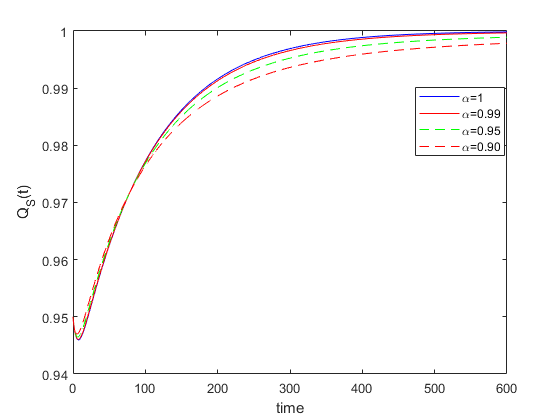}
	\caption{Solution trajectories for susceptible individuals illustrating asymptotic stability of the disease-free equilibrium $H_{df}$ with fractional orders $\alpha \in \{1,\ 0.99,\ 0.95,\ 0.90\}$ and $R_0 <1.$}
	\label{fg1}
\end{figure}

\begin{figure}[!htbp]
	\centering
	\includegraphics[scale=0.7]{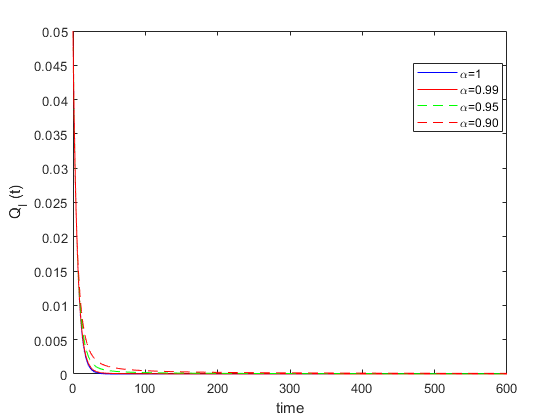}
	\caption{Solution trajectories for infected individuals illustrating asymptotic stability of the disease-free equilibrium $H_{df}$ with fractional orders $\alpha \in \{1,\ 0.99,\ 0.95,\ 0.90\}$ and $R_0<1.$}
	\label{fg2}
\end{figure}

\newpage
\begin{figure}[!htbp]%% using [h] forces the figure to be 'HERE'
%\centering
\subfigure[]{
\includegraphics[scale=0.6]{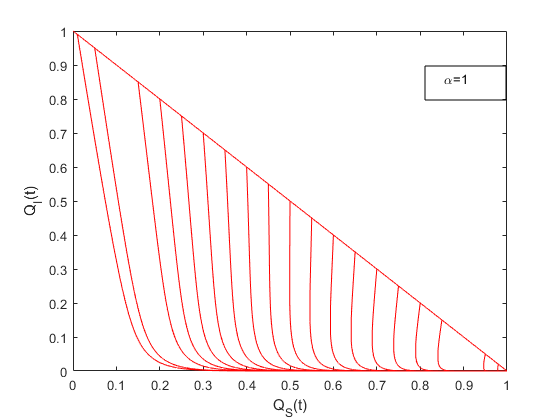}}\hfil
\subfigure[]{
\includegraphics[scale=0.6]{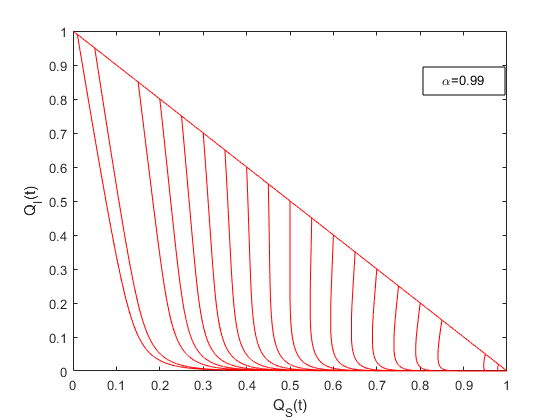}}\hfil
\subfigure[]{
\includegraphics[scale=0.6]{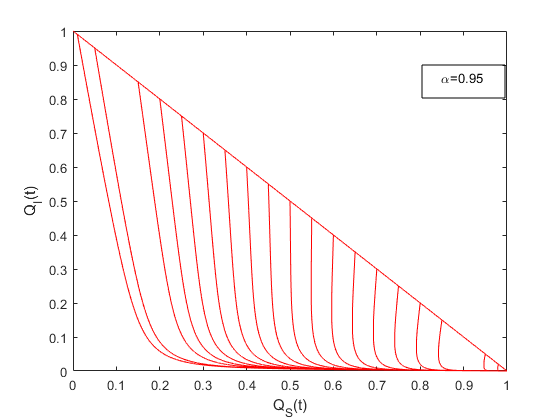}}\hfil
% Essaysurvc.eps: 300dpi, width=4.27cm, height=6.10cm, bb=50 50 554 770
\subfigure[]{\includegraphics[scale=0.6]{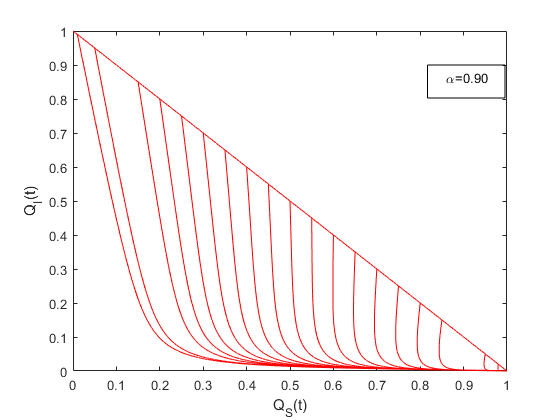}}
\caption{These graphs represent SI plane phase portraits for the SIS model with fractional orders $\alpha \in \{1,\ 0.99,\ 0.95,\ 0.90\}$ and $R_0<1.$}
\label{fg3}
\end{figure}

\newpage
\begin{figure}[!htbp]
	\centering
	\includegraphics[scale=0.7]{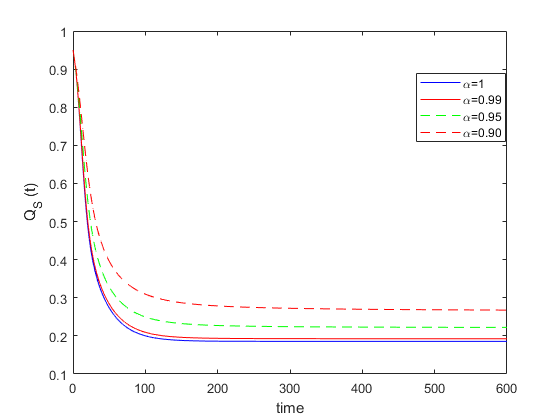}
	\caption{Solution trajectories for susceptible individuals illustrating asymptotic stability of the endemic equilibrium $H_{en}$ with fractional orders $\alpha \in \{1,\ 0.99,\ 0.95,\ 0.90\}$ and $R_0 >1.$}
	\label{fg4}
\end{figure}

\begin{figure}[!htbp]
	\centering
	\includegraphics[scale=0.7]{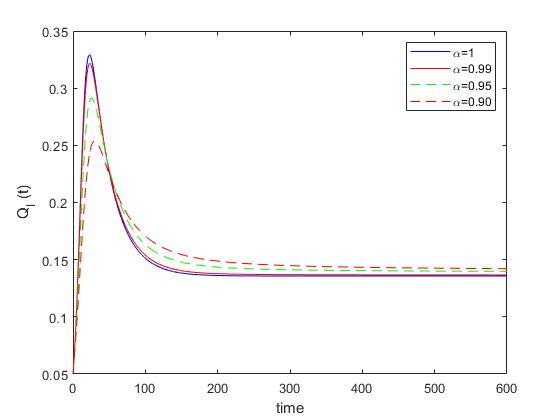}
	\caption{Solution trajectories for infected individuals illustrating asymptotic stability of the disease-free equilibrium $H_{en}$ with fractional orders $\alpha \in \{1,\ 0.99,\ 0.95,\ 0.90\}$ and $R_0> 1.$}
	\label{fg5}
\end{figure}

\newpage
\begin{figure}[!htbp]%% using [h] forces the figure to be 'HERE'
%\centering
\subfigure[]{
\includegraphics[scale=0.6]{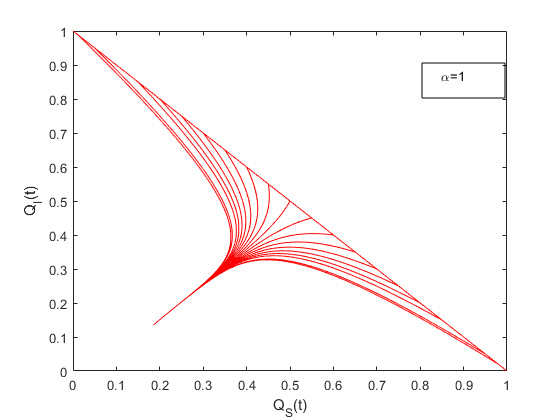}}\hfil
\subfigure[]{
\includegraphics[scale=0.6]{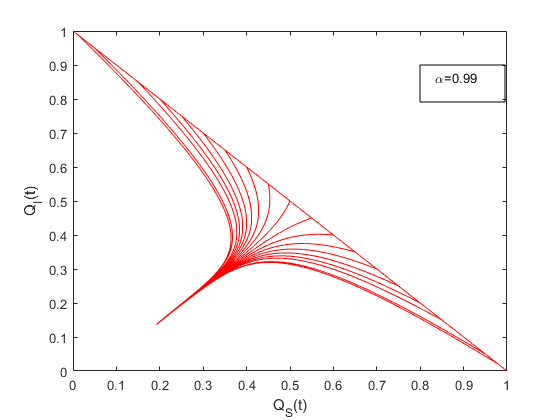}}\hfil
\subfigure[]{
\includegraphics[scale=0.6]{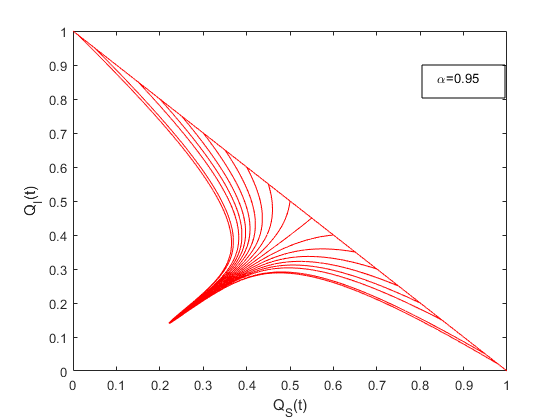}}\hfil
% Essaysurvc.eps: 300dpi, width=4.27cm, height=6.10cm, bb=50 50 554 770
\subfigure[]{\includegraphics[scale=0.6]{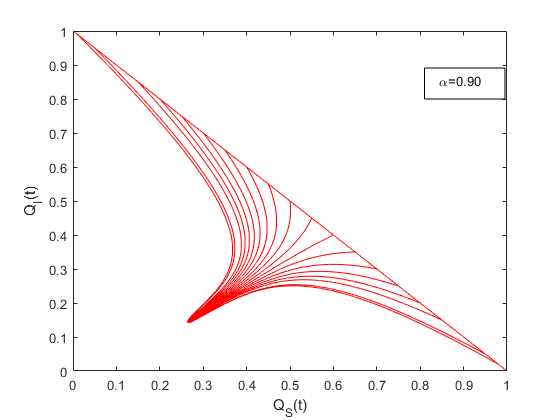}}
\caption{These graphs represent SI plane Phase portraits for the SIS model with fractional orders $\alpha \in \{1,\ 0.99,\ 0.95,\ 0.90\}$ and $R_0>1.$}
\label{fg6}
\end{figure}

\newpage
\section{SIRS model with fractional order dynamics}\label{foursec}
In this section, we introduce a new SIRS model that is characterised by system of nonlinear fractional order equations with standard incidence rate. This type of compartmental structure (susceptible-infected-recovered-susceptible) is used to construct and formulate nonlinear dynamical models for infectious diseases that that do not confer permanent immunity. The total population is grouped into three sub-populations namely susceptible individuals, $Q_S$, infected individuals, $Q_I$ and recovered individuals, $Q_R$. The integer order SIRS model studied in \citep{Vargas2011, mena1992dynamic} is given by

\begin{align}\label{SIRSMODEL_intger}
\frac{dQ_S(t)}{dt}&=\Lambda-\dfrac{\varphi Q_S Q_I}{Q_S+Q_I+Q_R}-\nu Q_S+ \gamma Q_R\nonumber\\\\ \nonumber
\frac{dQ_I(t)}{dt}&=\dfrac{\varphi Q_S Q_I}{Q_S+Q_I+Q_R}-\big(\delta+\kappa+\nu \big)Q_I\\ \nonumber\\
\frac{dQ_R(t)}{dt}&= \kappa Q_I-(\nu +\gamma)Q_R\nonumber \nonumber
\end{align}

where the positive model parameters which include $\Lambda,\ \varphi,\ \nu,\   \delta,\ \kappa$ and $\gamma$ represent recruitment rate, infection rate, natural death rate, recovered rate, disease induced death rate, and the loss of immunity rate respectively.\\

Motivated by the deterministic model~(\ref{SIRSMODEL_intger}), analyzed by the authors \citep{Vargas2011, mena1992dynamic} and the fractional order modeling approach proposed in \citep{wow9}, the new fractional order dynamical system we consider in this study is given by

\begin{align}\label{SIRSMODEL}
D_{t}^{\alpha}Q_S(t)&=\Lambda^{\alpha}-\dfrac{\varphi^{\alpha} Q_S Q_I}{Q_S+Q_I+Q_R}-\nu^{\alpha}Q_S+ \gamma^{\alpha}Q_R\nonumber\\\\ \nonumber
D_{t}^{\alpha}Q_I(t)&=\dfrac{\varphi^{\alpha} Q_S Q_I}{Q_S+Q_I+Q_R}-\big(\delta^{\alpha}+\kappa^{\alpha}+\nu^{\alpha}\big)Q_I\\ \nonumber\\
D_{t}^{\alpha}Q_R(t)&= \kappa^{\alpha}Q_I-(\nu^{\alpha}+\gamma^{\alpha})Q_R\nonumber \nonumber
\end{align}

Knowing that $N=Q_{S}+Q_{I}+Q_{R}$, the fractional order equation describing the dynamics of the total population is given by

\begin{align}\label{modeltotal2}
D_{t}^{\alpha}N &= \Lambda^{\alpha}- \delta^{\alpha} Q_{I}-\nu^{\alpha} N
\end{align}

Since the right-hand side of equation~(\ref{modeltotal2}) is not equal to zero, it follows that the total population may vary in time.

\newpage
\subsection{Equilibrium Points and local stability analysis}
The disease-free and endemic equilibrium points resulting from equating the right-hand side of the dynamical system~(\ref{SIRSMODEL}) to zero and solving for the states variables $Q_S,\ \ Q_I$ and $Q_R$ yields $P_{df}=\bigg(\frac{\Lambda^{\alpha}}{\nu^{\alpha}},0,0\bigg)$ and
$P_{en}=\big({Q_S}^{*}, {Q_I}^{*}, {Q_R}^{*}\big)$ respectively, where\\

 \begin{align*}
 {Q_S}^{*}&=\dfrac{\Lambda^{\alpha}\big(\gamma^{\alpha}+\kappa^{\alpha}+\nu^{\alpha}\big)}{\delta^{\alpha}\big(\gamma^{\alpha}+\nu^{\alpha}\big)\big(R_0-1\big)+\nu^{\alpha}\big(\gamma^{\alpha}+\kappa^{\alpha}+\nu^{\alpha}\big)R_0}\\\\
 {Q_I}^{*}&=\dfrac{\Lambda^{\alpha}\big(\gamma^{\alpha}+\nu^{\alpha}\big)\big(R_0-1\big)}{\delta^{\alpha}\big(\gamma^{\alpha}+\nu^{\alpha}\big)\big(R_0-1\big)+\nu^{\alpha}\big(\gamma^{\alpha}+\kappa^{\alpha}+\nu^{\alpha}\big)R_0}\\\\
 {Q_R}^{*}&=\dfrac{\Lambda^{\alpha}\kappa^{\alpha}\big(R_0-1\big)}{\delta^{\alpha}\big(\gamma^{\alpha}+\nu^{\alpha}\big)\big(R_0-1\big)+\nu^{\alpha}\big(\gamma^{\alpha}+\kappa^{\alpha}+\nu^{\alpha}\big)R_0}
 \end{align*}\\

The basic reproduction number $R_0$, corresponding to the fractional dynamical model~(\ref{SIRSMODEL}) is given by

$$R_0=\dfrac{\varphi^{\alpha}}{\big(\delta^{\alpha}+\kappa^{\alpha}+\nu^{\alpha}\big)}$$

\begin{theorem}
The equilibrium point $P_{df}$ (disease-free) of the SIRS model~(\ref{SIRSMODEL}) is locally asymptotically stable if
$R_{0}<1$ and unstable if $R_{0}>1.$
\end{theorem}

\begin{proof}
The Jacobian matrix for the dynamical fractional SIRS model computed at the equilibrium, $P_{df}$ is given by

\[J(P_{df})=
 \begin{bmatrix}
 -\nu^{\alpha}&& -\varphi^{\alpha}&& \gamma^{\alpha}\\\\
 0&& \varphi^{\alpha}-\big(\delta^{\alpha}+\kappa^{\alpha}+\nu^{\alpha}\big)&&0\\\\\\
 0&& \kappa^{\alpha}&& -\big(\nu^{\alpha}+\gamma^{\alpha}\big)
 \end{bmatrix}\]

For the point $P_{df}$ to be locally asymptotically stable, it is necessary and sufficient to show that the all the eigenvalues of matrix $J(H_{df})$ satisfy the stability condition~(\ref{condition}) constructed and studied in \citep{Ahmed2006}.

\newpage
From the $3\times 3$ Jacobian matrix, $J(P_{df})$, it is clear that one of the eigenvalues has negative real part ($\lambda_{1}= -\nu^{\alpha} < 0$) and the remaining two eigenvalues can obtained from the $2\times 2$ sub-matrix $E$ given below

 \[E=
 \begin{bmatrix}
 \varphi^{\alpha}-\big(\delta^{\alpha}+\kappa^{\alpha}+\nu^{\alpha}\big)&& 0 \\\\\\
  \kappa^{\alpha} && -\big(\nu^{\alpha}+\gamma^{\alpha}\big)
 \end{bmatrix}
 \]

The remaining eigenvalues will also the satisfy the stability condition~(\ref{condition}),
provided the Routh-Hurwitz stability conditions for the characteristic equation of matrix $E$ are satisfied.\\

The characteristic equation of  the $2\times 2 $ sub-matrix $E$ is given by

\begin{equation}\label{chart}
\lambda^2+a_1 \lambda+a_2=0
\end{equation}

where

  \begin{align*}
  a_{1}&= -\big[\varphi^{\alpha}-\big(\delta^{\alpha}+\kappa^{\alpha}+\nu^{\alpha}\big)\big] +\nu^{\alpha}+\gamma^{\alpha}\\
  &= -\big(\delta^{\alpha}+\kappa^{\alpha}+\nu^{\alpha}\big)\bigg[\dfrac{\varphi^{\alpha}}{\delta^{\alpha}+\kappa^{\alpha}+\nu^{\alpha}}-1\bigg] + \nu^{\alpha}+\gamma^{\alpha}\\\\
  &=\big(\delta^{\alpha}+\kappa^{\alpha}+\nu^{\alpha}\big)\bigg[1-\dfrac{\varphi^{\alpha}}{\delta^{\alpha}+\kappa^{\alpha}+\nu^{\alpha}}\bigg] + \nu^{\alpha}+\gamma^{\alpha}\\\\
  &=\big(\delta^{\alpha}+\kappa^{\alpha}+\nu^{\alpha}\big)\big[1-R_{0}\big] + \nu^{\alpha}+\gamma^{\alpha} >0  \hspace{0.3cm} if \hspace{0.3cm} R_{0} <1\\\\
  a_{2}&=\big(-\nu^{\alpha}-\gamma^{\alpha}\big)\big[\varphi^{\alpha}-\big(\delta^{\alpha}+\kappa^{\alpha}+\nu^{\alpha}\big)\big]\\
  &=-\big(\nu^{\alpha}+\gamma^{\alpha}\big)\big(\delta^{\alpha}+\kappa^{\alpha}+\nu^{\alpha}\big)\bigg[\dfrac{\varphi^{\alpha}}{\delta^{\alpha}+\kappa^{\alpha}+\nu^{\alpha}}-1\bigg] \\\\
  &=\big(\nu^{\alpha}+\gamma^{\alpha}\big)\big(\delta^{\alpha}+\kappa^{\alpha}+\nu^{\alpha}\big)\big[1-R_{0}\big] >0  \hspace{0.3cm}if \hspace{0.3cm} R_0 <1
  \end{align*}

Since the Routh-Hurwitz stability conditions ($a_1>0, \ a_2>0,$) have been verified, it follows that the two eigenvalues will have negative real parts.
\end{proof}

\newpage
We now consider local stability of the equilibrium point $P_{en}$. The Jacobian matrix of the fractional order model~(\ref{SIRSMODEL}) evaluated at $P_{en}$ is given as follows\\

 \[J(P_{en})=
 \begin{bmatrix}
 -\nu^{\alpha}-\frac{\varphi^{\alpha}{Q_I}^{*}\big({Q_I}^{*}+{Q_R}^{*}\big)}{\big({Q_S}^{*}+{Q_I}^{*}+{Q_R}^{*}\big)^{2}}&&-\frac{\varphi^{\alpha}{Q_S}^{*}\big({Q_S}^{*}+{Q_R}^{*}\big)}{\big({Q_S}^{*}+{Q_S}^{*}+{Q_R}^{*}\big)^{2}}&& \gamma^{\alpha}+\frac{\varphi^{\alpha}{Q_S}^{*}{Q_I}^{*}}{\big({Q_S}^{*}+{Q_I}^{*}+{Q_R}^{*}\big)^{2}}\\\\\\
 \frac{\varphi^{\alpha}{Q_I}^{*}\big({Q_I}^{*}+{Q_R}^{*}\big)}{\big({Q_S}^{*}+{Q_I}^{*}+{Q_R}^{*}\big)^{2}}&&\frac{\varphi^{\alpha}{Q_S}^{*}\big({Q_S}^{*}+{Q_R}^{*}\big)}{\big({Q_S}^{*}+{Q_I}^{*}+{Q_R}^{*}\big)}-\big(\delta^{\alpha}+\kappa^{\alpha}+\nu^{\alpha}\big)&&- \frac{\varphi^{\alpha}{Q_S}^{*}{Q_I}^{*}}{\big({Q_S}^{*}+{Q_I}^{*}+{Q_R}^{*}\big)^{2}}\\\\\\
 0&& \kappa^{\alpha} && -\big(\nu^{\alpha}+\gamma^{\alpha}\big)
 \end{bmatrix}
 \]

 Using the identity $\dfrac{\varphi^{\alpha} {Q_S}^{*}}{{Q_S}^{*}+{Q_I}^{*}+{Q_R}^{*}}=\big(\delta^{\alpha}+\kappa^{\alpha}+\nu^{\alpha}\big)$, the Jacobian matrix can be simplified as

 \[ J(P_{en})=
 \begin{bmatrix}
 -\nu^{\alpha}-\frac{\varphi^{\alpha}{Q_I}^{*}\big({Q_I}^{*}+{Q_R}^{*}\big)}{\big({Q_S}^{*}+{Q_I}^{*}+{Q_R}^{*}\big)^{2}}&&-\frac{\varphi^{\alpha}{Q_S}^{*}\big({Q_S}^{*}+{Q_R}^{*}\big)}{\big({Q_S}^{*}+{Q_I}^{*}+{Q_R}^{*}\big)^{2}}&& \gamma^{\alpha}+\frac{\varphi^{\alpha}{Q_S}^{*}{Q_I}^{*}}{\big({Q_S}^{*}+{Q_I}^{*}+{Q_R}^{*}\big)^{2}}\\\\
 \frac{\varphi^{\alpha}{Q_I}^{*}\big({Q_I}^{*}+{Q_R}^{*}\big)}{\big({Q_S}^{*}+{Q_I}^{*}+{Q_R}^{*}\big)^{2}}&&-\frac{\varphi^{\alpha}{Q_S}^{*}{Q_I}^{*}}{\big({Q_S}^{*}+{Q_I}^{*}+{Q_R}^{*}\big)^{2}}&&- \frac{\varphi^{\alpha}{Q_S}^{*}{Q_I}^{*}}{\big({Q_S}^{*}+{Q_I}^{*}+{Q_R}^{*}\big)^{2}}\\\\
 0&& \kappa^{\alpha} && -\big(\nu^{\alpha}+\gamma^{\alpha}\big)
 \end{bmatrix}
 \]

\newpage
 The characteristic equation of matrix $J(P_{en})$ is given as follows
	\begin{equation}\label{eqn11}
	\lambda^3+w_1\lambda^2+w_2\lambda+w_3=0,
	\end{equation}

 where
 \begin{align*}
 w_{1}&=\frac{\varphi^{\alpha}{Q_I}^{*}\big({Q_I}^{*}+{Q_R}^{*}\big)}{\big({Q_S}^{*}+{Q_I}^{*}+{Q_R}^{*}\big)^{2}}      +\frac{\varphi^{\alpha}{Q_S}^{*}{Q_I}^{*}}{\big({Q_S}^{*}+{Q_I}^{*}+{Q_R}^{*}\big)^{2}} +2\nu^{\alpha}+\gamma^{\alpha}\\\\
 w_{2}&=\dfrac{(\varphi^{\alpha})^2 {Q_I}^{*}{Q_S}^{*}\big({Q_I}^{*}+{Q_R}^{*}\big)\big({Q_S}^{*}+{Q_R}^{*}\big)}{\big({Q_S}^{*}+{Q_I}^{*}+{Q_R}^{*}\big)^{4}} + \frac{(\varphi^{\alpha})^2 {Q_S}^{*}({Q_I}^{*})^2 \big({Q_I}^{*}+{Q_R}^{*}\big)}{\big({Q_S}^{*}+{Q_I}^{*}+{Q_R}^{*}\big)^{4}} \\&+ \frac{\kappa^{\alpha}\varphi^{\alpha}{Q_S}^{*}{Q_I}^{*}}{\big({Q_S}^{*}+{Q_I}^{*}+{Q_R}^{*}\big)^{2}} + \frac{\nu^{\alpha}\varphi^{\alpha}{Q_I}^{*}\big({Q_I}^{*}+{Q_R}^{*}\big)}{\big({Q_S}^{*}+{Q_I}^{*}+{Q_R}^{*}\big)^{2}} + \frac{2\nu^{\alpha}\varphi^{\alpha}{Q_S}^{*}{Q_I}^{*}}{\big({Q_S}^{*}+{Q_I}^{*}+{Q_R}^{*}\big)^{2}} \\& + \frac{\gamma^{\alpha}\varphi^{\alpha}{Q_I}^{*}\big({Q_I}^{*}+{Q_R}^{*}\big)}{\big({Q_S}^{*}+{Q_I}^{*}+{Q_R}^{*}\big)^{2}} + \frac{\gamma^{\alpha}\varphi^{\alpha}{Q_S}^{*}{Q_I}^{*}}{\big({Q_S}^{*}+{Q_I}^{*}+{Q_R}^{*}\big)^{2}} + \nu^{\alpha}\gamma^{\alpha}+ (\nu^{\alpha})^{2}\\ \\
 w_{3}&=\frac{\nu^{\alpha}(\varphi^{\alpha})^2{Q_S}^{*}{Q_I}^{*}\big({Q_I}^{*}+{Q_R}^{*}\big)\big({Q_S}^{*}+{Q_R}^{*}\big)}{\big({Q_S}^{*}+{Q_I}^{*}+{Q_R}^{*}\big)^{4}} + \frac{\nu^{\alpha}(\varphi^{\alpha})^2({Q_I}^{*})^2 {Q_S}^{*}\big({Q_I}^{*}+{Q_R}^{*}\big)}{\big({Q_S}^{*}+{Q_I}^{*}+{Q_R}^{*}\big)^{4}}\\& +
  \frac{\gamma^{\alpha}(\varphi^{\alpha})^2{Q_S}^{*}{Q_I}^{*}\big({Q_I}^{*}+{Q_R}^{*}\big)\big({Q_S}^{*}+{Q_R}^{*}\big)}{\big({Q_S}^{*}+{Q_I}^{*}+{Q_R}^{*}\big)^{4}} + \frac{\gamma^{\alpha}(\varphi^{\alpha})^2({Q_I}^{*})^2 {Q_S}^{*}\big({Q_I}^{*}+{Q_R}^{*}\big)}{\big({Q_S}^{*}+{Q_I}^{*}+{Q_R}^{*}\big)^{4}} \\&+
  \frac{\kappa^{\alpha}\nu^{\alpha}\varphi^{\alpha}{Q_S}^{*}{Q_I}^{*}}{\big({Q_S}^{*}+{Q_I}^{*}+{Q_R}^{*}\big)^{2}} + \frac{\gamma^{\alpha}\nu^{\alpha}\varphi^{\alpha}{Q_S}^{*}{Q_I}^{*}}{\big({Q_S}^{*}+{Q_I}^{*}+{Q_R}^{*}\big)^{2}} - \frac{\kappa^{\alpha}\gamma^{\alpha}\varphi^{\alpha}{Q_I}^{*}\big({Q_I}^{*}+{Q_R}^{*}\big)}{\big({Q_S}^{*}+{Q_I}^{*}+{Q_R}^{*}\big)^{2}}\\&+
  \frac{(\nu^{\alpha})^2\varphi^{\alpha}{Q_S}^{*}{Q_I}^{*}}{\big({Q_S}^{*}+{Q_I}^{*}+{Q_R}^{*}\big)^{2}}
 \end{align*}

 Let $D(q)$ represent  the discriminant of a polynomial function $q$ given by
 $$q(x)=x^3+w_1x^2+w_2x+w_3$$

  then
	
\begin{equation}\label{eqn13}
	D(q)=
	\begin{vmatrix}
	1&w_1&w_2&w_3&0\\
	0&1&w_1&w_2&w_3\\
	3&2w_1&w_2&0&0\\
	0&3&2w_1&w_2&0\\
	0&0&3&2w_1&w_2
	\end{vmatrix}
	=18w_1 w_2 w_3+(w_1 w_2)^2-4 w_3 w_1^2-4 w_2^2-27 w_3^2.
	\end{equation}

\newpage
Motivated by the fractional order stability approach in \citep{Ahmed2006} and the recent applications of this concept in
\citep{silva2019stability, kilicman2018fractional, Salman2017, Pinto2016}, we obtain the proposition below

\begin{proposition}
	One assume that $P_{en}$ exists in $\mathbb{R}^3_+.$
	\begin{enumerate}[(i).]
%\item If $R_0>1,$ then the equilibrium point $P_{en}$ is locally asymptotically stable.
\item If $D(q)>0$  and the Routh-Hurwitz condition are satisfied, i.e.,
     $w_1>0$, $w_3>0$, $w_1$ $w_2>w_3$, then $P_{en}$ is locally asymptotically stable.
\item If $D(q)<0$, $w_1\geq0$, $w_2\geq0$, $w_3>0$, $\alpha<2/3 $ then $P_{en}$ is locally asymptotically stable.
\item If $D(q)<0$, $w_1<0$, $w_2<0$, $\alpha>2/3 $ then all roots of $q(\lambda)=0$ satisfy the stability condition $\abs{arg(\lambda)} >\frac{\alpha \pi}{2}$.
\item If $D(q)<0$, $w_1>0$, $w_2>0$, $w_1 w_2=w_3$  $\forall \alpha\in(0,1]$ then $P_{en}$ is locally asymptotically stable.
\item $w_3>0$ is a necessary condition for $P_{en}$ to be locally asymptotically stable.
	\end{enumerate}
\end{proposition}

 \subsection{Numerical simulations and discussions}
This subsection deals with numerical approximation of the new SIRS fractional epidemic model~(\ref{SIRSMODEL}). For our numerical illustrations, we have applied the fracPECE iterative scheme developed by \cite{Diethelm1999}. We have also used the Matlab code fde12.m designed by \cite{Garrappa2011} for the fracPECE iterative scheme.\\

To demonstrate local asymptotic stability of the disease-free point, we considered the parameter values: $\Lambda=0.01,\  \beta=0.06,\ \nu=0.01,\   \delta=0.15,\  \kappa=0.3,\ \gamma=0.02$ and initial conditions $Q_{S0}=0.95,\ Q_{I0}=0.05,\ Q_{R0}=0$. Local asymptotic stability of the endemic point is illustrated by setting $\Lambda=0.01,\  \beta=0.5,\ \nu=0.01,\  \delta=0.015,\  \kappa=0.2,\ \gamma=0.02$ and initial conditions $Q_{S0}=0.95,\ Q_{I0}=0.05,\ Q_{R0}=0$. Figures~\ref{fg7}, \ref{fg8} and \ref{fg9} shows trajectories of susceptible, infected  and recovered individuals respectively for different values of $\alpha$ illustrating asymptotic stability of the disease-free equilibrium. Local asymptotic stability of the endemic equilibrium is demonstrated in Figures~\ref{fg11}, \ref{fg12} and \ref{fg13}. The sub-plots in Figures~\ref{fg10} and \ref{fg14} illustrate SI plane phase portraits for the SIRS model.

\begin{figure}[!htbp]
	\centering
	\includegraphics[scale=0.7]{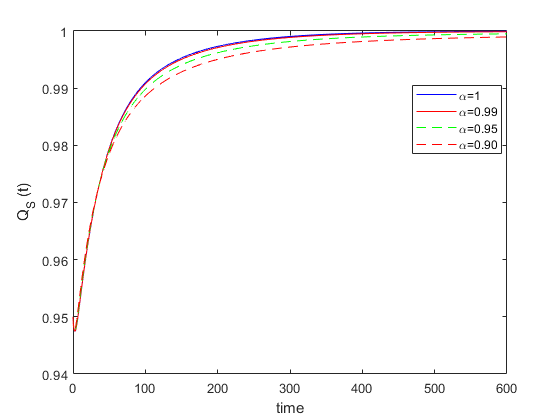}
	\caption{Solution trajectories for susceptible individuals illustrating asymptotic stability of the disease-free equilibrium $H_{df}$ with fractional orders $\alpha \in \{1,\ 0.99,\ 0.95,\ 0.90\}$ and $R_0 <1.$}
	\label{fg7}
\end{figure}

\begin{figure}[!htbp]
	\centering
	\includegraphics[scale=0.7]{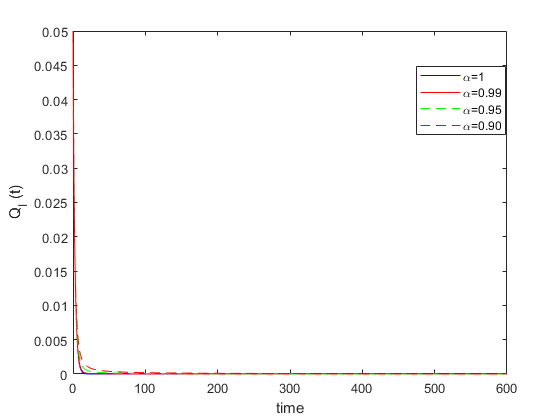}
	\caption{Solution trajectories for infected individuals illustrating asymptotic stability of the disease-free equilibrium $H_{df}$ with fractional orders $\alpha \in \{1,\ 0.99,\ 0.95,\ 0.90\}$ and $R_0 <1.$}
	\label{fg8}
\end{figure}

\newpage
\begin{figure}[!htbp]
	\centering
	\includegraphics[scale=0.7]{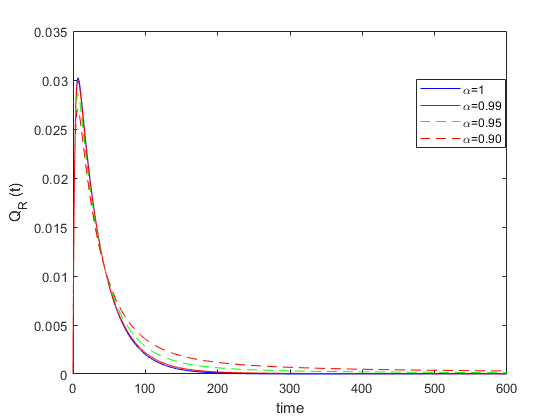}
	\caption{Solution trajectories for recovered individuals illustrating asymptotic stability of the disease-free equilibrium $H_{df}$ with fractional orders $\alpha \in \{1,\ 0.99,\ 0.95,\ 0.90\}$ and $R_0 <1.$}
	\label{fg9}
\end{figure}

\newpage
\begin{figure}[!htbp]%% using [h] forces the figure to be 'HERE'
%\centering
\subfigure[]{
\includegraphics[scale=0.6]{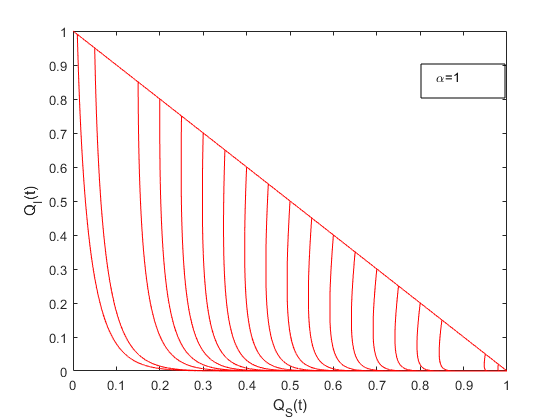}}\hfil
\subfigure[]{
\includegraphics[scale=0.6]{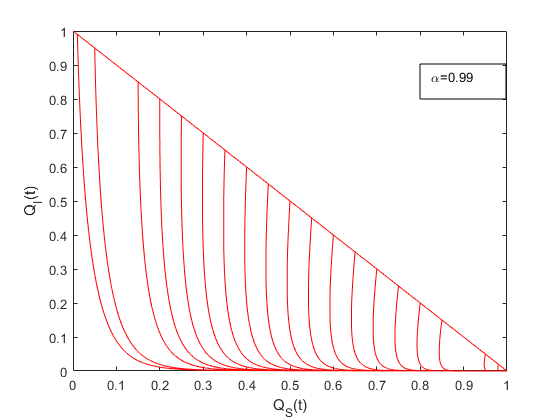}}\hfil
\subfigure[]{
\includegraphics[scale=0.6]{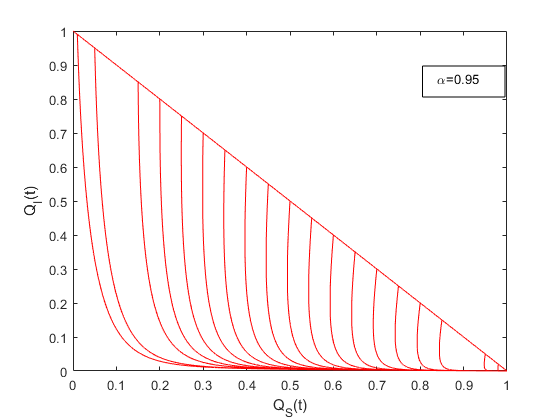}}\hfil
% Essaysurvc.eps: 300dpi, width=4.27cm, height=6.10cm, bb=50 50 554 770
\subfigure[]{\includegraphics[scale=0.6]{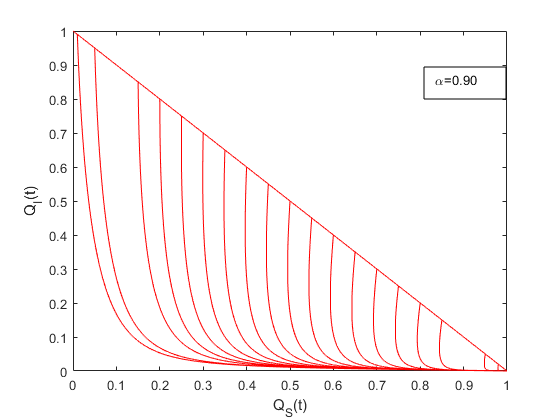}}
\caption{These sub-plots represent SI plane Phase portraits for the SIRS model with fractional orders $\alpha \in \{1,\ 0.99,\ 0.95,\ 0.90\}$ and $R_0<1.$}
\label{fg10}
\end{figure}

\newpage
\begin{figure}[!htbp]
	\centering
	\includegraphics[scale=0.7]{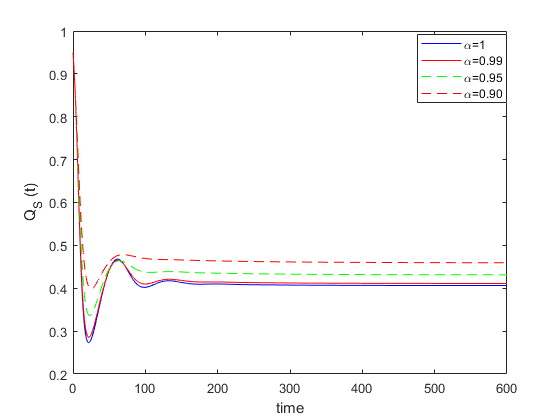}
	\caption{Solution trajectories for susceptible individuals illustrating asymptotic stability of the endemic equilibrium $P_{en}$ with fractional orders $\alpha \in \{1,\ 0.99,\ 0.95,\ 0.90\}$ and $R_0 >1.$}
	\label{fg11}
\end{figure}

\begin{figure}[!htbp]
	\centering
	\includegraphics[scale=0.7]{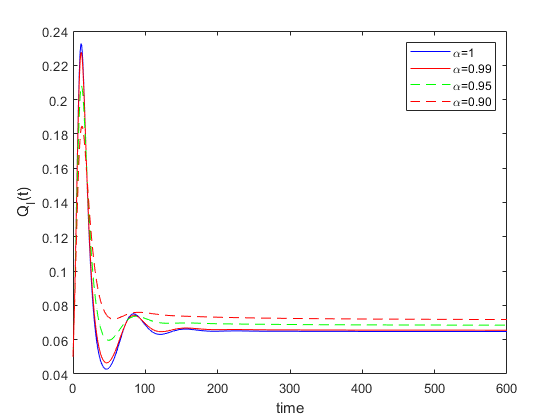}
	\caption{Solution trajectories for infected individuals illustrating asymptotic stability of the endemic equilibrium $P_{en}$ with fractional orders $\alpha \in \{1,\ 0.99,\ 0.95,\ 0.90\}$ and $R_0 >1.$}
	\label{fg12}
\end{figure}

\newpage
\begin{figure}[!htbp]
	\centering
	\includegraphics[scale=0.7]{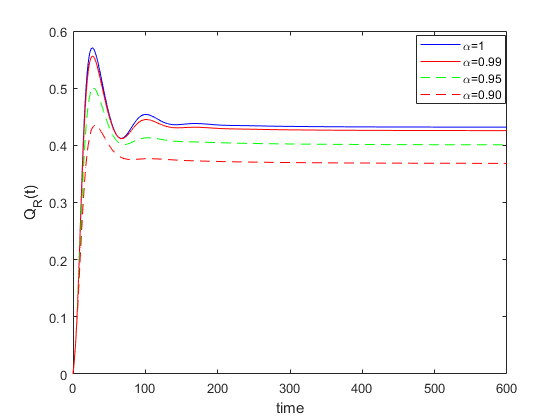}
	\caption{Solution trajectories for recovered individuals illustrating asymptotic stability of the endemic equilibrium $P_{en}$ with fractional orders $\alpha \in \{1,\ 0.99,\ 0.95,\ 0.90\}$ and $R_0 >1.$}
	\label{fg13}
\end{figure}

\newpage
\begin{figure}[!htbp]%% using [h] forces the figure to be 'HERE'
%\centering
\subfigure[]{
\includegraphics[scale=0.6]{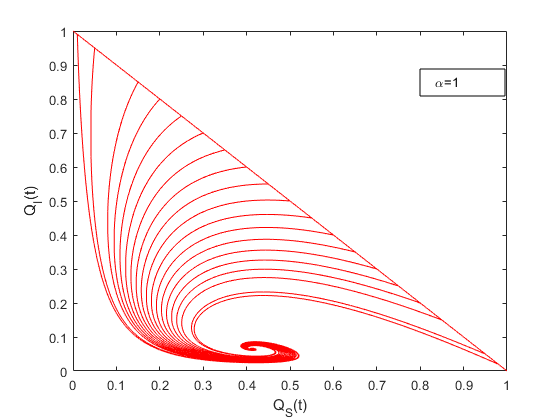}}\hfil
\subfigure[]{
\includegraphics[scale=0.6]{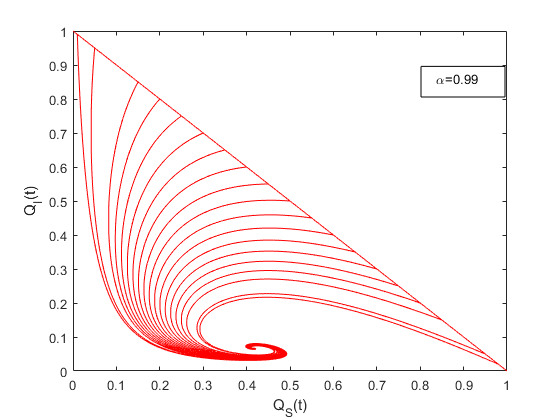}}\hfil
\subfigure[]{
\includegraphics[scale=0.6]{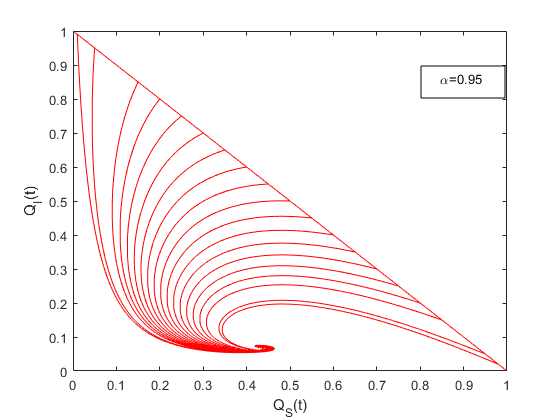}}\hfil
% Essaysurvc.eps: 300dpi, width=4.27cm, height=6.10cm, bb=50 50 554 770
\subfigure[]{\includegraphics[scale=0.6]{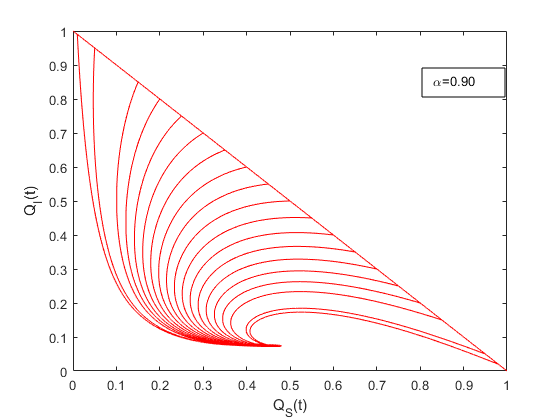}}
\caption{These sub-plots represent SI plane phase portraits for the SIRS model with fractional orders $\alpha \in \{1,\ 0.99,\ 0.95,\ 0.90\}$ and $R_0>1.$}
\label{fg14}
\end{figure}

 \section{Conclusion}\label{consec}
 In this work, we have presented and analyzed two epidemic models that are constructed with Caputo fractional derivative. Our fractional order modelling construction was motivated by the approach introduced by the author in \citep{wow9}. We have studied a modified fractional order SIS model from an existing mathematical model. A new Caputo based SIRS model is also proposed and analyzed. In both mathematical models, we have analytically and numerically studied local asymptotic stability of the equilibrium points (disease-free and endemic). Our numerical examples suggest that, fractional order models are appropriate for compartmental modeling of infectious diseases dynamics and can yield interesting results.

 \bibliography{sample2}
\end{document}